\documentclass[11pt]{article}
\usepackage{amsmath,amsthm,amsfonts,amssymb,amscd, amsxtra}
\usepackage{color}
\newtheorem{theorem}{Theorem}
\newtheorem{lemma}[theorem]{Lemma}
\newtheorem{definition}{Definition}
\newtheorem{corollary}[theorem]{Corollary}
\newtheorem{proposition}[theorem]{Proposition}

\newtheorem{remark}{Remark}

\begin{document}
\title{A  Newton conditional gradient method for constrained nonlinear systems }

\author{
    Max L.N. Gon\c calves
    \thanks{Institute of Mathematics and Statistics, Federal University of Goias, Campus II- Caixa
    Postal 131, CEP 74001-970, Goi\^ania-GO, Brazil. (E-mails: {\tt
       maxlng@ufg.br} and {\tt jefferson@ufg.br}). The work of these authors was
    supported in part by  CNPq Grant  444134/2014-0, 309370/2014-0, 200852/2014-0  and FAPEG/GO.}
    \and
      Jefferson G. Melo \footnotemark[1]
}
 \maketitle
\begin{abstract}
In this paper, we consider the problem of solving a constrained system of nonlinear equations. 
We propose an algorithm based on a combination of the Newton and conditional gradient methods, and establish its local convergence analysis. Our analysis is set up by using a  majorant condition technique, allowing us to prove in a unified way convergence results for  two large families of nonlinear functions. 
The first one includes functions whose derivative satisfies a H\"{o}lder-like condition, and the second one consists of a substantial subclass of analytic functions. 
Numerical experiments illustrating the applicability of the proposed method are presented, and  comparisons  with some other methods are discussed.

\end{abstract}

\noindent {{\bf Keywords:} constrained nonlinear systems; Newton method; conditional gradient method;  local
convergence.}

\maketitle
\section{Introduction}\label{sec:int}
In this paper, we consider the problem of finding a solution of the constrained  system of nonlinear equations
\begin{equation}\label{eq:p}
F(x)=0, \qquad x\in C,
\end{equation}
where  $F:\Omega \to \mathbb{R}^n $ is a continuously differentiable nonlinear function and $\Omega \subset \mathbb{R}^n$ is an open set  containing the nonempty  convex compact set $C$.  
 This problem  appears in many application areas such as engineering,  chemistry and  economy.  The constraint set may naturally arise in order to exclude  solutions of the model with no physical meaning, or it may be considered artificially due to some knowledge about the problem itself (see, for example,   \cite{morini1,Mangasarian1,sandra} and references therein). 
Different approaches to solve \eqref{eq:p} have been proposed in the literature. Many of them are related to their unconstrained counterpart having as focus the Newton Methods whenever applicable. Strategies based on different techniques such as trust region,  active set and gradient methods have also been used; see, for instance, \cite{morini1,morini2,Mangasarian1,Kan2,sandra,marinez2,Porcelli,Zhang1,Zhu2005343}.

  In this paper, we propose a Newton conditional gradient (Newton-CondG) method  for solving~\eqref{eq:p}, which consists of  a Newton step followed by a procedure related to the  conditional gradient (CondG) method. 
 The procedure plays the role of  getting the Newton iterate back to the constraint set
 in such a way that  the fast convergence of Newton method is maintained.
The  CondG  method (also known as Frank-Wolfe method \cite{Dunn1980,NAV:NAV3800030109}) requires at each iteration the minimization of a linear functional over the feasible constraint set.
This requirement is considered relatively simple and can be  fulfilled efficiently for many problems. Moreover, depending on the problem and the structure of the constraint set, linear optimization oracles may provide solutions with specific characteristics leading to important properties such as sparsity and low-rank,  see, e.g., \cite{Freund2014,ICML2013_jaggi13} for a discussion on this subject.
Due to these facts  and its simplicity,  CondG method  have recently received a lot of attention from a theoretical and computational point of view, see for instances
   \cite{Freund2014,nemiro2014,ICML2013_jaggi13,lan2015,Ya-2015,Ronny2013} and references therein. An interesting approach is to combine variants of  CondG method with  some superior well designed algorithms;
   for instance,  augmented Lagrangian   and accelerated gradient methods, see \cite{lan2015,Ya-2015}.
   In this sense, our combination of  CondG and Newton methods seems to be  promising.

 We present a local convergence analysis of Newton-CondG method. 
More specifically, we  provide an estimate of the  convergence radius,
for which the well-definedness and the convergence of the method are ensured. 
Furthermore, results on  convergence rates of the method are also established.
 Our analysis is done via the concept of majorant condition which, 
 besides improving the convergence theory, allows  to study  Newton type methods in a unified way, see   \cite{F10,MAX1,MAX3,FS06}  and references therein.
 Thus, our analysis covers  two large families of nonlinear functions,  namely,    one  satisfying a H\"{o}lder-like condition,
which includes functions with Lipschitz derivative, and  another one  satisfying a Smale condition, which includes a substantial class of analytic functions. 
 Finally, we also present some numerical experiments illustrating the applicability of our method and discuss its behavior compared with other methods.

This paper is organized as follows.  In Section \ref{sec:PMF},  we study a certain scalar sequence generated by a Newton-type method. The Newton-CondG method and its convergence analysis are discussed in  
 Section \ref{sec:condGmet}. 
Section \ref{sec:HolderSmale} specializes our main convergence result for functions satisfying  H\"{o}lder-like  and Smale conditions. 
In Section \ref{NunEx}, we present some numerical experiments illustrating the applicability of the proposed method.

\vspace{2mm}
\noindent { \bf Notations and basic assumptions:}  
Throughout this paper, we assume that $F:\Omega \to \mathbb{R}^n $ is a continuously differentiable nonlinear function
where  $\Omega \subset \mathbb{R}^n$ is an open set containing nonempty  convex compact set $C$.
The  Jacobian matrix  of $F$ at $x\in \Omega$ is denoted by $F'(x)$. 
 We also assume that there exists  $x_*\in C$  such that $F(x_*)=0$ and $F'(x_*)$ is nonsingular.
 Let  the inner product and its associated norm in $\mathbb{R}^n$ be denoted by $\langle\cdot,\cdot\rangle$ and $\| \cdot \|$, respectively. The open ball
centered at $a \in \mathbb{R}^n$ and radius $\delta>0$ is denoted  by $B(a,\delta)$.
For a given linear operator $T:\mathbb{R}^n\to \mathbb{R}^n$, we also use  $\| \cdot \|$ to denote its norm, which is defined by  
$$
\|T\|:=\sup\{ \|Tx\|,\;\|x\|\leq 1\}.
$$
\section{Preliminaries results} \label{sec:PMF}

 Our goal in this section is to  study the behavior of a scalar sequence generated by a Newton-type method applied to solve 
$$f(t)=0,$$  
where $f:[0,\; R)\to \mathbb{R}$ is a continuously differentiable function  such that
\begin{itemize}
  \item[{\bf h1.}]  $f(0)=0$ and $f'(0)=-1$;
  \item[{\bf  h2.}]  $f'$ is  strictly increasing.
\end{itemize}  

Although {\bf h1} implies  that $t_*=0$ is a solution of the above equation, 
the convergence properties of this scalar sequence will be directly associated to 
the sequence generated by Newton-CondG method.  
First, consider the scalar  $\nu$   given by
\begin{equation}\label{nu}
 \nu:=\sup\{t\in [0, R): f'(t)<0\}.
\end{equation}
 Since $f'$ is continuous  and  $f'(0)=-1$, it follows  that $\nu>0$. Moreover,   {\bf h2} implies that $f'(t)<0$ for all $t\in [0,\nu)$. Hence, the following Newton iteration map for  $f$ is well defined:
\begin{equation} \label{eq:def.nf}
  \begin{array}{rcl}
  n_f:[0,\, \nu)&\to& \mathbb{R}\\
    t&\mapsto& t-f(t)/f'(t).
  \end{array}
\end{equation}
Let us also consider  the scalars $\lambda$ and $\rho$ such that
 \begin{equation}\label{rho}
\lambda\in [0,1),\quad  \rho:=\sup\left\{\delta \in(0, \nu):(1+\lambda)\frac{|n_f(t)|}{t}+\lambda<1, \; t\in(0, \delta)\right\}. 
\end{equation}
We now present  some properties of the Newton iteration map $n_f$   and show that $\rho> 0$.

\begin{proposition}  \label{pr:incr1}
The following statements hold:
\begin{itemize}
  \item[{ a)}]  $n_f(t)<0$ for all $t \in\,  (0,\,\nu)$;
  \item[{  b)}]  $\lim_{t\downarrow 0}|n_f(t)|/t=0;$
    \item[{  c)}] the scalar $\rho$ is positive and 
\begin{equation}\label{eq:001}
0<(1+\lambda)|n_f(t)|+t\lambda<t,\qquad \forall t\in (0, \, \rho).
\end{equation}
    \end{itemize}  
\end{proposition}
\begin{proof}
(a) From condition {\bf h2} we see that  $f'$ is strictly increasing in $[0,R)$, in particular,  $f$ is strictly convex. Hence,  since $\nu\leq R$ (see \eqref{nu}),  we obtain $f(0)>f(t)+f'(t)(0-t),$ for any $t\in  (0,\, \nu)$ 
which combined with $f(0)=0$ and  $f'(t)<0$ for any $t\in (0, \nu)$, proves  item (a).

(b) In view of item (a) and the fact that $f(0)=0$, we obtain
\begin{equation} \label{eq:rho}
\frac{|n_f(t)|}{t}= \frac{1}{t}\left(\frac{f(t)}{f'(t)}-t\right)=\frac{1}{f'(t)} \frac{f(t)-f(0)}{t-0}-1, \qquad \forall t\in (0,\,\nu).
\end{equation}
As  $f'(0)\neq 0$, item (b) follows by taking limit in~\eqref{eq:rho}, as $t$  $\downarrow 0$.

(c) Since $0\leq\lambda<1$, using items~(a) and (b),
 we conclude that there exists 
$\delta>0$ such that
$$
0<\frac{|n_f(t)|}{t}<\frac{1-\lambda}{1+\lambda},
 \qquad   \forall t\in \, (0, \delta),
$$
or, equivalently,
$$
0<(1+\lambda)\frac{|n_f(t)|}{t}+\lambda<1,\qquad  \forall t\in\,  (0, \delta).
$$
Therefore,   $\rho$  is positive and \eqref{eq:001} trivially holds.
\end{proof}
Let  $t_0\in(0,\rho)$ and $\{\theta_k\}\subset [0,+\infty)$  be given, and  define the scalar sequence $\{t_k\}$ by 
\begin{equation} \label{eq:tknk}
t_{k+1}=(1+\sqrt{2\theta_k})|n_f(t_k)|+\sqrt{2\theta_k} t_k, \qquad \forall k\geq 0.
\end{equation}

\begin{corollary} \label{cr:kanttk}
Assume that $\{\theta_k\} \subset [0,\lambda^2/2]$. Then
 the sequence $\{t_k\}$ is well defined,  strictly decreasing and  converges to $0$. Moreover, 
$ \limsup_{k\to \infty}t_{k+1}/t_k=\sqrt{2\tilde \theta},$ where $ \tilde \theta=\limsup_{k\to \infty}{\theta_k}$.
\end{corollary}
\begin{proof}
First of all,  since  $(0,\rho)\subset \mbox{dom}(n_f)$,  in order to show the well definedness of  $\{t_k\}$ is sufficient to prove  that  $t_k\in (0, \rho)$ for all $k$. Let us prove this latter  inclusion by induction on $k$.
As $t_0\in(0,\rho)$, the statement trivially  holds for $k=0$. Now, 
assume that it holds for some $k\geq0$.
Hence, $t_k \in \mbox{dom}(n_f)$ and  it follows from Proposition~ \ref{pr:incr1} (a) and \eqref{eq:tknk}   that
 \begin{equation}\label{aj34}
 0< t_{k+1}= (1+{\sqrt{2\theta_k}})|n_f(t_k)|+{\sqrt{2\theta_k}}t_k\leq (1+\lambda)|n_f(t_k)|+\lambda t_k< t_k,
 \end{equation}
where the second and third inequalities are due to $0\leq\sqrt{2\theta_k}\leq \lambda$ and  \eqref{eq:001}, respectively. It follows from \eqref{aj34} and the induction assumption that $t_{k+1}\in (0,\rho)$, concluding the induction proof. Thus,  $\{t_k\}$ is well defined 
and  \eqref{aj34} also implies that
it is strictly decreasing.
As a consequence, we have 
 $\{t_k \}$ converges to some $t_* \in [0,\rho)$. Thus,  since $n_f(\cdot)$ is continuous, taking 
 limit superior in \eqref{aj34} as $k \to \infty$, we obtain, in particular,  
 $$
  t_{*}\leq (1+\lambda)|n_f(t_*)|+\lambda t_*.
 $$
Therefore,  \eqref{eq:001}  implies that $t_*=0$. Now, using 
 \eqref{eq:tknk},  $\lim_{k\to \infty}t_{k}=0$  and Proposition~\ref{pr:incr1}~(b), we have
   $$
\displaystyle   \limsup_{k\to \infty}\frac{t_{k+1}}{t_k}=\limsup_{k\to \infty} \,\left((1+\sqrt{2\theta_k})\frac{|n_f(t_k)|}{t_k}+\sqrt{2\theta_k}\right)=\sqrt{2\tilde \theta},
   $$
concluding the proof of the corollary.
\end{proof}

\section{The  method and its convergence analysis}\label{sec:condGmet}

In this section, we  propose  a Newton conditional gradient method  to solve \eqref{eq:p} and 
discuss its  local convergence results. 

\subsection{The Newton-CondG  method}\label{sec:method}
In this subsection, we   present  a method for solving \eqref{eq:p} which consists of  a Newton step followed by a procedure related to an inexact  conditional gradient method. This procedure is used in order to retrieve the Newton iterate back to the constraint set $C$ in such a way that the fast convergence of the sequence generated by the method is ensured. We assume the existence of a linear optimization oracle (LO oracle) capable of minimizing linear functions over $C$. 
The Newton conditional gradient method is formally described as follows.

\noindent
\\
\hrule
\noindent
\\
{\bf  Newton-CondG method\\} \label{CGM}
\hrule
\begin{description}

\item[ Step 0.] Let $x_0\in C$ and   $\{\theta_j\}\subset[0,\infty)$   be given and set $k=0$.
\item[ Step 1.] If $F(x_k)=0$, then {\bf stop}; otherwise, compute $s_k \in \mathbb{R}^n$ and $y_k\in \mathbb{R}^n$ such that
\begin{equation}\label{aa0}
F'(x_k) s_k=-F(x_k),  \quad y_{k}=x_k+s_k.
\end{equation}
\item[ Step 2.]  Given $\theta_k\geq 0$, use CondG procedure  to obtain $x_{k+1}\in C$ 
as
\begin{equation}\label{aa1}
x_{k+1}=\mbox{CondG}(y_{k},x_k,\theta_k\|s_k\|^2).
\end{equation}
\item[ Step 3.]  Set $k\gets k+1$, and go to step~1.
\end{description}
\noindent
{\bf end}\\
\noindent
\\
{ {\bf CondG procedure} $z=\mbox{CondG}(y,x,\varepsilon)$} \label{CGM}
\begin{description}
\item[ P0.]   Set $z_1=x$ and  $t=1$.
\item[ P1.] Use the LO oracle to compute an optimal solution $u_t $ of  
$$
g_{t}^*=\min_{u \in  C}\{\langle  z_t-y,u-z_t \rangle\}.
$$
\item[ P2.] If $ g^*_{t}\geq -\varepsilon $,  set $z=z_t$ and {\bf stop} the procedure; otherwise, compute $\alpha_t \in \, (0,1]$ and $z_{t+1}$ as
 
$$
{\alpha}_t: =\min\left\{1, \frac{-g^*_{t}}{\|u_t-z_t\|^2}  \right\}, \qquad  z_{t+1}=z_t+ \alpha_t(u_t-z_t).
$$
\item[ P3.] Set $t\gets t+1$, and go to {\bf P1}. 
\end{description}
{\bf end procedure}\\
\hrule
\noindent
\\
  Some remarks regarding Newton-CondG method are in order. First, in step~1, we check whether the current iterate $x_k$ is  a solution, and if not, we  execute a Newton step.  
   Second, since the iterate $y_k$ obtained after a Newton step  may be infeasible to the constraint set $C$, we apply an inexact conditional gradient method in order to get the new iterate $x_{k+1}$ in $C$. As mentioned before,   this method requires an oracle which is assumed to be able to minimize linear functions over the constraint set. It is clear that this may be done efficiently for a wide class of sets, since many methods for minimizing linear functions are well established in the literature. Third,  if CondG procedure computes  $g^*_t\geq -\varepsilon $ then it stops and out put  $z_t \in C$. Hence, if the procedure  continues, we have $g^*_t<-\varepsilon \leq0$ which implies that the stepize $\alpha_t$ is well defined and belongs to  $ (0,1]$.

\subsection{ Convergence of  Newton-CondG method} \label{sec:proof}

In this subsection,  we discuss the convergence behavior of  Newton-CondG method.  First, we present
the concept of majorant functions and some of their properties. Then,  some properties of CondG procedure are studied. Finally, we  state and prove our main result.
 Basically, our main theorem specifies a convergence radius of the method and  analyze its convergence  results.
Moreover, it also present the relationship between the Newton-CondG sequence $\{x_k\}$ and sequence $\{t_k\}$ defined in~\eqref{eq:tknk}, which will  be associated to our majorant function.

We start by defining, for any given $R\in(0,+\infty]$,  the scalar $\kappa$ as
 \begin{equation}\label{kappa}
\kappa:=\kappa(\Omega,R)=\sup \left\{ t\in [0, R): B(x_*, t)\subset \Omega \right\}.
\end{equation} 
 
In order to analyze  the convergence properties of Newton-CondG method, we consider  the concept of majorant functions which has the advantage of  presenting, in a unified way, convergence result for different classes of nonlinear functions; more details about the majorant condition can be found in \cite{F10,MAX1,MAX3,FS06}.

\begin{definition}\label{deff:maj_funct}
Let $R\in(0,+\infty]$ be given and consider $\kappa$ as in \eqref{kappa}.
A function  $f:[0,\; R)\to \mathbb{R}$ continuously differentiable is said to be a majorant function for the function $F$ on 
$B(x_*,\kappa)$ if and only if
  \begin{equation}\label{Hyp:MH}
\left\|F'(x_*)^{-1}\left[F'(x)-F'(x_*+\tau(x-x_*))\right]\right\| \leq    f'\left(\|x-x_*\|\right)-f'\left(\tau\|x-x_*\|\right),
  \end{equation}
  for all $\tau \in [0,1]$ and $x\in B(x_*, \kappa)$, and  conditions   {\bf h1} and 
  {\bf h2}   are satisfied.

\end{definition}

To illustrate  Definition~\ref{deff:maj_funct},
let $\mathcal{L}$ be the class  of continuously differentiable functions $G:\mathbb{R}^n \to \mathbb{R}^n$ 
such that $G '(x_*)$ is non-singular and the following H\"{o}lder type condition is satisfied  
$$
 \left\|G'(x_*)^{-1}(G'(x)-G'(x_*+\tau(x-x_*)))\right\|\leq  K(1-\tau^p) \|x-x_*\|^p, \;   x\in \mathbb{R}^n, \;\tau \in [0,1],
$$
for some $K>0$ and $ 0< p \leq 1$. It is easy to see that the function    \mbox{$f:[0, \infty) \to \mathbb {R}$}    given by  $f(t)=(K/(p+1))t^{p+1}-t$ is a majorant function  for any  $G \in \mathcal{L}$. It is worth pointing out  that the class $\mathcal{L}$ includes, in particular, functions $G$ with Lipschitz continuous derivative such that $G '(x_*)$ is non-singular. Section \ref{sec:HolderSmale} contains more details on this class of functions and also another one for analytic functions satisfying a Smale condition.

Before presenting some properties of majorant functions, let  $R\in(0,+\infty]$ be given and  $\kappa$ as defined in \eqref{kappa}, and consider the following condition:

\begin{itemize}
 \item[{\bf  A1.}] the function $F$ has a majorant function    $f:[0,\; R)\to \mathbb{R}$  on $B(x_*,\kappa)$.
\end{itemize}

Let us now present a result which is fundamental for the convergence analysis of Newton-CondG method. More precisely, it 
highlights the relationship between the nonlinear function $F$ and its majorant function $f$.

\begin{lemma} \label{lemgeral:relateF_nf} 
Assume  that {\bf A1} holds, and  let  $x \in  B\left(x^*,\min\{\kappa,\nu\}\right)$,
 where $\nu$ is defined in \eqref{nu}. Then the function $F'(x)$ is invertible  and the following estimates hold:
\begin{itemize}
\item [a)]  $\|F'(x)^{-1}F'(x_*)\|\leqslant  1/|f'(\| x-x_*\|)|;$
\item [b)] $\|F'(x)^{-1}F(x)\|\leq f(\|x-x_*\|)/ f' (\|x-x_*\|);$ 
\item [c)] $\|F'(x_*)^{-1}\left[ F(x_*)-F(x)-F'(x)(x_*-x)\right]\|\leq  f' (\|x-x_*\|)\|x-x_*\|-f(\|x-x_*\|)$.
\end{itemize}
\end{lemma}
\begin{proof}
The proof follows the same pattern as the proofs of Lemmas 10, 11 and 12 in~\cite{MAX1}.
\end{proof}

Next, we present two results which contain some basic properties of  CondG procedure.


\begin{lemma}  \label{pr:condi}
For any $y,\tilde y\in \mathbb{R}^n$, $x, \tilde x \in C$  and $\mu \geq 0$, we have
$$
\| \mbox{CondG}(y,x,\mu) - \mbox{CondG}(\tilde y,\tilde x,0)\|\leq  \|y-\tilde y\|+ \sqrt{2\mu}. 
$$
\end{lemma}
\begin{proof}
Let us denote $z=\mbox{CondG}(y,x,\mu)$ and $\tilde z=\mbox{CondG}(\tilde y,\tilde x,0)$.
It follows from CondG procedure  that  $z,\tilde z \in C$ and 
\begin{equation}\label{eq:deg}
\langle z-y,\tilde z-z\rangle\geq -\mu,\qquad \langle \tilde z-\tilde y,z-\tilde z\rangle\geq 0.
\end{equation}
On the other hand,  after some simple algebraic manipulations we have
$$\|y-\tilde y\|^2=\|z-\tilde z\|^2+\|(y-z)-(\tilde y-\tilde z)\|^2+2\langle y-z-(\tilde y-\tilde z),z-\tilde z\rangle,$$
which implies that
$$\|z-\tilde z\|^2\leq \|y-\tilde y\|^2+2\langle z-y,z-\tilde z\rangle+2\langle \tilde y-\tilde z ,z-\tilde z \rangle.$$
The last inequality together with \eqref{eq:deg} yields 
$$\|z-\tilde z\|^2\leq \|y-\tilde y\|^2+2\mu,$$
and then
$$\|z-\tilde z\|\leq \|y-\tilde y\|+\sqrt{2\mu},$$
 which combined with the definitions of $z$ and  $\tilde z$ proves the lemma. 
 \end{proof}
The following scalar will be used in the convergence analysis of Newton-CondG method:
\begin{equation}\label{r}
 r:=\min\{\rho,\kappa\},
\end{equation}
where $\rho$ and $\kappa$ are defined in~\eqref{rho} and \eqref{kappa}, respectively.
Since Lemma \ref{lemgeral:relateF_nf} implies that  $F'(x)$ is invertible for any $x\in B(x_*, r)$, we see that the following Newton iteration map is well defined:
\begin{equation} \label{NF}
  \begin{array}{rcl}
  N_{F}:B(x_*, r) &\to& \mathbb{R}^n\\
    x&\mapsto& x-S_F(x),
  \end{array} 
\end{equation}
where
\begin{equation} \label{SF}
S_F(x):=F'(x)^{-1}F(x).
\end{equation}
\begin{lemma} \label{l:wdef}
Assume that {\bf A1} holds, and let  $x \in C\cap B(x^*,r)$ and $\theta\geq 0$. Then  there holds
\begin{equation}\label{o698}
  \|\mbox{CondG}(N_F(x),x,\theta\|S_{F}(x)\|^2) -x_*\| \leq(1+\sqrt{2\theta})|n_f(\|x-x_*\|)|+\sqrt{2\theta}\|x-x_*\|,
\end{equation}
 where $n_f$ is defined in \eqref{eq:def.nf}. 
As a consequence, letting $\lambda$ as in \eqref{rho},  if  $\sqrt{2\theta}\leq \lambda$, then
\begin{equation}\label{eq:contractionCondG}
\mbox{CondG}(N_F(x),x,\theta\|S_{F}(x)\|^2)\in C\cap B(x^*,r).
\end{equation}
\end{lemma}
\begin{proof}
It follows from Lemma \ref{pr:condi} with $y=N_F(x)$,  $\tilde y=N_F(x_*)$,  $\tilde x=x_*$ and 
$\mu=\theta\|S_{F}(x)\|^2$ that
\[
  \| \mbox{CondG}(N_F(x),x,\theta\|S_{F}(x)\|^2)-\mbox{CondG}(N_F(x^*),x^*,0)\|\leq  \|N_F(x)-N_F(x_*)\|+ \sqrt{2\theta}{\|S_F(x)\|}.
\]
It is easy to see from CondG procedure that  $\mbox{CondG}(x,x,0)=x,$ for all $x\in C$.
Hence,  since $F(x_*)=0$ implies that $N_F(x_*)=x_*$, we have $ \mbox{CondG}(N_F(x^*),x^*,0)=x^*$. 
Thus,  the last inequality gives
\begin{equation}\label{eq:127890}
  \| \mbox{CondG}(N_F(x),x,\theta\|S_{F}(x)\|^2)-x_{*}\|\leq  \|N_F(x)-x_*\|+ \sqrt{2\theta}{\|S_F(x)\|}.
\end{equation}
On the other hand, using \eqref{NF} and $F(x_*)=0$, we have
$$
N_F(x)-x_*=F'(x)^{-1}\left[ F(x_*)-F(x)-F'(x)(x_*-x)\right],
$$
which combined with Lemma \ref{lemgeral:relateF_nf} (a) and (c), and the fact that $f'(\|x-x_*\|)<0$ gives
$$
 \|N_F(x)-x_*\|\leq\frac{ f' (\|x-x_*\|)\|x-x_*\|-f(\|x-x_*\|)}{|f'(\|x-x_*\|)|}=\frac{ f (\|x-x_*\|)}{f' (\|x-x_*\|)}-\|x-x_*\|.$$
Hence, it follows from  \eqref{eq:127890},  \eqref{SF} and Lemma~\ref{lemgeral:relateF_nf}(b) that
\begin{align*}
  \| \mbox{CondG}(N_F(x),x,\theta\|S_{F}(x)\|^2)-x_{*}\|&\leq  (1+\sqrt{2\theta})\frac{ f (\|x-x_*\|)}{f' (\|x-x_*\|)}-\|x-x_*\|\\
 &= (1+\sqrt{2\theta})|n_f(\|x-x_*\|)|+\sqrt{2\theta}\|x-x_*\|,
\end{align*}
where the equality is due to definition of $n_f$ in \eqref{eq:def.nf} and Proposition~\ref{pr:incr1}~(a). Thus, \eqref{o698} is proved.

Now, since $\sqrt{2\theta}\leq \lambda$ and $0<\|x-x_*\|<r\leq \rho$, it follows from  \eqref{eq:001} with $t=\|x-x_*\|$ that
$$
(1+\sqrt{2\theta})|n_f(\|x-x_*\|)|+\sqrt{2\theta}\|x-x_*\|<\|x-x_*\|<r,
 $$
 which together with \eqref{o698} gives $\mbox{CondG}(N_F(x),x,\theta\|S_{F}(x)\|^2)\in  B(x^*,r)$.
 Therefore, \eqref{eq:contractionCondG} follows from the fact that points generated by CondG procedure belong to $C$. 
 \end{proof}

In the following, we state and  prove our main convergence result.

\begin{theorem}\label{th:nt} 
Let $\lambda, \rho$ and $\kappa$ be 
as  in~\eqref{rho} and \eqref{kappa}, and  consider 
\begin{equation}\label{radius}
r=\min\{\rho,\kappa\}.
\end{equation}
Assume that  {\bf A1} holds, and 
let  $\{\theta_k\}$ and $x_0$ be given in step~0 of Newton-CondG method. If $\{\theta_k\} \subset [0,\lambda^2/2]$ and   $x_0\in C\cap B(x_*, r)\backslash \{x_*\}$, then   Newton-CondG method  generates a sequence  $\{x_k\}$ which is
contained in $B(x_*,r)\cap C$, converges to  $x_*$ and satisfies
\begin{equation}\label{eq:xk_linearconv}
    \limsup_{k \to \infty} \; [{\|x_{k+1}-x_*\|}\big{/}{\|x_k-x_*\|}]\leq\sqrt{2\tilde\theta},
\end{equation}
where $\tilde \theta=\limsup_{k\to \infty}\theta_k$.
Moreover, given $0\leq p\leq1$ and $n_f$ as in \eqref{eq:def.nf}, if  the following assumption holds
\begin{itemize}
  \item[{\bf  h3.}] the function  $(0,\, \nu) \ni t \mapsto |n_f(t)|/t^{p+1}$ is  strictly increasing;
\end{itemize}
then, for all integer $k\geq 0$, we have
\begin{equation}\label{eq:xkispconvergent}
\|x_{k+1}-x_*\| \leq
(1+\lambda)\frac{|n_f(t_0)|}{{t_0}^{p+1}}\|x_k-x_*\|^{p+1}+\lambda\|x_k-x_*\|,
\end{equation}
and
\begin{equation}\label{eq:tk_majorizes_xk}
\|x_{k}-x_*\|\leq t_k, 
\end{equation}
where $\{t_k\}$ is as defined in \eqref{eq:tknk} with $t_0=\|x_0-x^*\|$.
\end{theorem}\begin{proof}
First of all,  it is easy to see from  \eqref{NF},  \eqref{SF} and Newton-CondG method that 
\begin{equation}\label{eq:239}
x_{k+1}=\mbox{CondG}(N_F(x_k),x_k,\theta_k\|S_F(x_k)\|^2).
\end{equation}
Hence, since   $x_0\in C\cap B(x_*, r)\backslash \{x_*\}$,  it follows from the first statement of Lemma~\ref{lemgeral:relateF_nf}, inclusion~\eqref{eq:contractionCondG} with $x=x_k$ and $\theta=\theta_k$, and  a simple induction argument that  Newton-CondG method  generates a sequence $\{x_k\}$  contained in $B(x_*,r)\cap C$.

We will now prove that  $\{x_k\}$ converges to  $x_*$.
Since for all  $k\geq 0$, $\|x_{k}-x_*\|<r\leq \rho$,  it follows from \eqref{eq:239} and
inequality~\eqref{o698} with  $x=x_{k}$ and $\theta=\theta_{k}$ that, for all $k\geq 0$,
\begin{equation}\label{eq:conv1}
\begin{array}{rcl}
\|x_{k+1}-x_*\| &\leq& (1+\sqrt{2\theta_k})|n_f(\|x_k-x_*\|)|+\sqrt{2\theta_k}\|x_k-x_*\|\\[2mm]
&\leq & (1+\lambda)|n_f(\|x_k-x_*\|)|+\lambda\|x_k-x_*\|\\[2mm]
&<& \|x_{k}-x_*\|,
\end{array}
\end{equation}
where the second and the third inequalities are due to  $\sqrt{2\theta_k}\leq \lambda$ and   \eqref{eq:001} with $t=\|x_k-x_*\|$, respectively. 
Hence, $\{\|x_{k}-x_*\| \}$ converges to some $\ell_* \in [0,\rho)$. Thus, as    $n_f(\cdot)$ is  continuous in~$[0, \rho)$,
 \eqref{eq:conv1} implies, in particular, that  
$$ \ell_{*} \leq (1+\lambda)|n_f(\ell_*)|+\lambda\ell_*.$$ 
Therefore, due to \eqref{eq:001}  we must have  $\ell_*=0$, proving that $x_k \rightarrow x_*$.

We also see from   \eqref{eq:conv1} that, for all $k \geq 0$, 
\[
\frac{\|x_{k+1}-x_*\|}{\|x_{k}-x_*\|}\leq (1+\sqrt{2\theta_k})\frac{|n_f(\|x_{k}-x_*\|)|}{\|x_{k}-x_*\|}+\sqrt{2\theta_k}.
\]
The asymptotic rate~\eqref{eq:xk_linearconv} follows by taking  limit superior in the last inequality as $k \to \infty$ and using  $\|x_{k}-x_*\|\to 0$, Proposition~\ref{pr:incr1}~(b) and 
$\limsup_{k\to \infty}\theta_k=\tilde \theta$.

In order to prove the second part of the theorem, let us assume that {\bf h3} holds.
In view of~\eqref{eq:conv1}, we have  $\|x_k-x_*\|\leq \|x_0-x^*\|=t_0$ and 
\begin{equation}\label{eq:conseqH3}
\|x_{k+1}-x_*\|\leq (1+\sqrt{2\theta_k})\frac{|n_f(\|x_k-x_*\|)|}{\|x_k-x_*\|^{p+1}}\|x_k-x_*\|^{p+1}+\sqrt{2\theta_k}\|x_k-x_*\|,\quad  \forall k\geq 0\,.
\end{equation}
Therefore,  \eqref{eq:xkispconvergent} follows from  assumption {\bf h3} and $\sqrt{2\theta_k}\leq \lambda$.

Let us now show inequality \eqref{eq:tk_majorizes_xk} by induction. Since $t_0=\|x_0-x_*\|$, it trivially holds for $k=0$. Assume that $\|x_{k}-x_*\|\leq t_k$ for some $k\geq 0$.  Hence, \eqref{eq:conseqH3} together with {\bf h3} yields
\begin{align*}
\|x_{k+1}-x_*\| &\leq (1+\sqrt{2\theta_k})\frac{|n_f(t_k)|}{t_k^{p+1}}\|x_k-x_*\|^{p+1}+\sqrt{2\theta_k}\|x_k-x_*\|\\
                       &\leq (1+\sqrt{2\theta_k}){|n_f(t_k)|}+\sqrt{2\theta_k} t_k=t_{k+1}.\end{align*}
Therefore,  inequality \eqref{eq:tk_majorizes_xk}  holds for $k+1$, concluding the proof.
\end{proof}
\begin{remark}\label{remark_maintheo}
It is worth mentioning that if the error sequence $\{\theta_k\}$ given in step~0 of Newton-CondG method converges to zero, then it follows from \eqref{eq:xk_linearconv} that the sequence $\{x_k\}$ generated by Newton-CondG method is superlinear convergent to $x^*$. Also, under assumption {\bf h3}, it follows from \eqref{eq:xkispconvergent} that if $\lambda=0$, then $\{\|x_{k+1}-x^*\|/ \|x_k-x^*\|^{p+1}\}$ is bounded, improving the superlinear convergence of $\{x_k\}$. However, we shall point out that $\lambda=0$ implies that 
$\theta_k=0$ for all $k$,  which in turn may impose  a stringent  stopping criterium for CondG procedure (see step {\bf P2} with $\varepsilon=0$).
 In the next section, we consider two classes of nonlinear functions whose  majorant functions satisfy {\bf h3}.
\end{remark}
\section{Convergence results  under H\"{o}lder-like and Smale  conditions} \label{sec:HolderSmale}

In this section, we  specialize Theorem \ref{th:nt}  for two classes of functions. In the first one, $F'$ satisfies a H\"{o}lder-like condition \cite{F10,huangy2004}, and in the second one,  $F$ is an analytic function  satisfying a Smale condition \cite{Shen201029,S86}.
\begin{theorem}\label{th:HV} 
Let  $\kappa=\kappa(\Omega,\infty)$ as defined in~\eqref{kappa}. Assume
  that there exist a constant $K>0$ and $ 0< p \leq 1$ such that
\begin{equation}\label{holder_cond}
\left\|F'(x_*)^{-1}(F'(x)-F'(x_*+\tau(x-x_*)))\right\|\leq  K(1-\tau^p) \|x-x_*\|^p, \quad   x\in B(x_*, \kappa), \quad \tau \in [0,1].
\end{equation}
Take $\lambda\in [0,1)$ and 
let $$\bar r:=\min \left\{\kappa, \left[\frac{(1-\lambda)(p+1)}{K(2p+1-\lambda)}\right]^{1/p}\right\}.$$
 If $\{\theta_k\} \subset [0,\lambda^2/2]$ and   $x_0\in C\cap B(x_*, \bar r)\backslash \{x_*\}$, then   Newton-CondG method  generates a sequence  $\{x_k\}$ which is
contained in $B(x_*,\bar r)\cap C$, converges to  $x_*$ and satisfies
$$
\|x_{k+1}-x_*\| \leq
\frac{(1+\lambda) pK}{(p+1)[1- K\,\|x_0-x_*\|^{p}]}\|x_k-x_*\|^{p+1}+\lambda\|x_k-x_*\|,
\qquad  \forall k\geq 0.
$$
Moreover, if  $t_0=\|x_0-x^*\|$, then there holds
$$
\|x_{k+1}-x_*\|\leq t_{k+1}:=\frac{(1+\lambda) \,p K t_{k}^{p+1}}{(p+1)[1-Kt_k^{p}]}+ \lambda, \qquad  \forall k\geq 0.
$$
\end{theorem}
\begin{proof}
It is easy to prove that  $f:[0, \infty)\to \mathbb{R}$ defined by
$
f(t)=Kt^{p+1}/(p+1)-t
$
is a majorant function for $F$ on $B(x_*,\kappa)$, which   satisfies   {\bf h3} in Theorem \ref{th:nt}.  Moreover,  in this case, it is easily seen that $\nu$ and $\rho$, as defined in  \eqref{nu} and \eqref{rho}, respectively,  satisfy
$$
\rho=\left[\frac{(1-\lambda)(p+1)}{K(2p+1-\lambda)}\right]^{1/p} < \nu=[1/ K]^{1/p},
$$
and, as a consequence,  $\bar r=\min \{\kappa,\; \rho\}=r$ (see \eqref{radius}). Therefore, the statements of the theorem follow from
 Theorem~\ref{th:nt}.
\end{proof}
\begin{remark}
As already mentioned, if the function $F$ has Lipschitz continuous derivative then \eqref{holder_cond} is satisfied with $p=1$, and then Theorem \ref{th:HV}  holds for such a class of nonlinear functions. Additionally to the assumptions  of Theorem \ref{th:HV}, if  $\{\theta_k\}$ converges to zero, then $\{x_k\}$ converges superlinear to $x^*$.
Moreover, if $\lambda=0$ (i.e., $\theta_k=0$ for all k), we obtain the quadratic convergence rate of $\{x_k\}$.
\end{remark}


In the previous theorem, we analyzed convergence of Newton-CondG method under H\"{o}lder-like condition. Next, we present a similar result for the class of analytic functions satisfying a Smale condition.
\begin{theorem}\label{theo:Smale}
Assume that $F:{\Omega}\to \mathbb{R}^{n}$ is an analytic function and 
\begin{equation} \label{eq:SmaleCond}
\gamma := \sup _{ n > 1 }\left\| \frac
{F'(x_*)^{-1}F^{(n)}(x_*)}{n !}\right\|^{1/(n-1)}<+\infty.
\end{equation}
Let   $\lambda\in [0,1)$ be given  and  compute 
$$
\bar r:=\min \left\{\kappa,\frac{ 5-3\lambda-\sqrt{(5-3\lambda)^2-8(1-\lambda)^2}}{4(1-\lambda)\gamma}\right\},
$$
where $\kappa=\kappa(\Omega,1/\gamma)$ is as defined in~\eqref{kappa}. 
If $\{\theta_k\} \subset [0,\lambda^2/2]$ and   $x_0\in C\cap B(x_*, \bar r)\backslash \{x_*\}$, then   Newton-CondG method  generates a sequence  $\{x_k\}$ which is
contained in $B(x_*,\bar r)\cap C$, converges to  $x_*$ and satisfies
$$
\|x_{k+1}-x_*\| \leq
\frac{\gamma  }{2(1- \gamma \|x_0-x_*\|)^{2}-1}\|x_k-x_*\|^{2}+\sqrt{2\theta_k}{\|x_k-x_*\|},
\qquad \forall k\geq 0.
$$
Moreover, if  $t_0=\|x_0-x^*\|$, then there holds
$$
\|x_{k+1}-x_*\|\leq t_{k+1}:=\frac{(1+\lambda) \gamma t_{k}^{2}}{2(1-\gamma t_k)^2-1}+\sqrt{2\theta_k} \lambda, \qquad \forall k\geq 0.
$$
\end{theorem}
\begin{proof}
Under the assumptions of the theorem,
 the  real function $f:[0,1/\gamma) \to \mathbb{R}$
defined by $ f(t)={t}/{(1-\gamma t)}-2t$ is a majorant function for $F$ on $B(x_*, 1/\gamma)$, 
see for instance, \cite[Theorem~14]{MAX1}. Since $f'$ is convex, it  satisfies   {\bf h3} in Theorem \ref{th:nt} with p=1, see  \cite[Proposition~7]{MAX1}.
Moreover,  in this case, it is easily seen that $\nu$ and $\rho$, as defined in  \eqref{nu} and \eqref{rho}, respectively,  satisfy
$$\rho=\frac{ 5-3\lambda-\sqrt{(5-3\lambda)^2-8(1-\lambda)^2}}{4(1-\lambda)\gamma}, \qquad \nu=\frac{\sqrt{2}-1}{\sqrt{2}\gamma}, \qquad \rho<\nu< \frac{1}{\gamma},
$$
and, as a consequence,  $\bar r=\min \{\kappa,\; \rho\}=r$  (see \eqref{radius}). Therefore, the statements of the theorem follow from
 Theorem~\ref{th:nt}.
\end{proof}

\section{Numerical experiments} \label{NunEx}

This section summarizes the results of the numerical experiments we carried out in order to verify the effectiveness of Newton-CondG method. 
 In the following experiments, we considered  25 box-constrained nonlinear systems, i.e.,  problem~\eqref{eq:p} with $ C=\{x \in \mathbb{R}^n: l \leq x \leq u\}, $
 where $l, u \in \mathbb{R}^n$.
  We analyze the  set of 25  problems specified  in Table~1. 
 These well-known problems come from  different  applications and some of them are considered challenging ones.

\begin{table}[h]
\centering
\caption{The box-constrained nonlinear systems considered}
\vspace{0.5cm}
{\small
\begin{tabular}{|c|c|c|}
\hline
Problem & Name and souce  & n  \\ 
\hline     
Pb 1  & Himmelblau function \cite[14.1.1]{Jones:2000}   & 2 \\
Pb 2  & Equilibrium Combustion \cite[14.1.2]{Jones:2000}& 5 \\
Pb 3  & Bullard-Biegler system \cite[14.1.3]{Jones:2000} & 2  \\
Pb 4  & Ferraris-Tronconi system \cite[14.1.4]{Jones:2000} & 2  \\ 
Pb 5  & Brown's almost linear system \cite[14.1.5]{Jones:2000} & 5  \\ 
Pb 6  & Robot kinematics problem \cite[14.1.6]{Jones:2000} & 8 \\ 
Pb 7  & Circuit design problem \cite[14.1.7]{Jones:2000} & 9 \\ 
Pb 8  & Series of CSTRs $R=0.935$ \cite[14.1.8]{Jones:2000} & 2 \\ 
Pb 9  & Series of CSTRs $R=0.940$ \cite[14.1.8]{Jones:2000} & 2  \\ 
Pb 10  & Series of CSTRs $R=0.945$ \cite[14.1.8]{Jones:2000} & 2  \\ 
Pb 11  & Series of CSTRs $R=0.950$ \cite[14.1.8]{Jones:2000} & 2  \\ 
Pb 12  & Series of CSTRs $R=0.955$ \cite[14.1.8]{Jones:2000} & 2 \\ 
Pb 13 & Series of CSTRs $R=0.960$ \cite[14.1.8]{Jones:2000} & 2  \\ 
Pb 14 & Series of CSTRs $R=0.965$ \cite[14.1.8]{Jones:2000} & 2  \\ 
Pb 15 & Series of CSTRs $R=0.970$ \cite[14.1.8]{Jones:2000} & 2 \\ 
Pb 16 & Series of CSTRs $R=0.975$ \cite[14.1.8]{Jones:2000} & 2 \\ 
Pb 17  & Series of CSTRs $R=0.980$ \cite[14.1.8]{Jones:2000} & 2  \\ 
Pb 18  & Series of CSTRs $R=0.985$ \cite[14.1.8]{Jones:2000} & 2  \\ 
Pb 19  & Series of CSTRs $R=0.990$ \cite[14.1.8]{Jones:2000} & 2  \\ 
Pb 20  & Series of CSTRs $R=0.995$ \cite[14.1.8]{Jones:2000} & 2  \\ 
Pb 21  & Chemical reaction problem \cite[Problem~5]{sandra} & 67  \\ 
Pb 22  & A Mildly-Nonlinear BVP \cite[Problem~7]{sandra} & 451  \\ 
Pb 23  & H-equation, $c=0.99$ \cite[Problem~4]{more} & 100  \\ 
Pb 24  & H-equation, $c=0.9999$ \cite[Problem~4]{more} & 100  \\ 
Pb 25  & A Two-bar Framework \cite[Problem~1]{sandra} & 5  \\ 
\hline  
\end{tabular}
}
\end{table}

The computational results were obtained using MATLAB R2015a on
 a 2.5 GHz intel Core~i5 with 4GB of RAM and OS X system. 
In our implementation, the Jacobian matrices were approximated by finite differences
and  the error parameter $\theta_k$ was set equal to $10^{-5}$ for all $k$.
Moreover, CondG Procedure   stopped when either the required accuracy  was obtained  or the maximum of $300$ iterations were performed.
In order to compare Newton-CondG with other methods, we decided to keep the stopping criteria  $\|F(x_k)\|_{\infty}\leq10^{-6}$ and a failure  was declared if the number of iterations was greater than $300$. Furthermore, we also tested the method with  initial points  $x_0= l+ 0.25\gamma (u -l) $ with $\gamma=1,2,3$. However, since the choice $\gamma=3$ corresponds to an initial point that is a solution of Pb5 and  the Jacobian matrices of Pb6 and Pb22 are singular at the initial point obtained with $\gamma=2$, we used $\gamma=2.5$ in these cases.

Table~2  shows the performance of Newton-CondG method  for solving  23 of the 25 problems considered.   The other two problems (Pb~3 and Pb~7)  do not appear in Table~2, because the  method was not able to solve   them for none of the three choices of initial points.
 In the table, ``$\gamma$" and ``Iter" are  the constant $\gamma$ used to compute initial point $x_0$ and the number of iterations of Newton-CondG method, respectively, ``Time"  is the CPU time in seconds  and  ``$\|F\|_{\infty}$" is the infinity norm of $F$ at the final iterate $x_k$. Finally, the symbol ``$*$" indicates a failure.

\begin{table}[h]
\centering
\caption{Perfomance of Newton-CondG method  for 23 problems described in Table~1 }
\vspace{0.5cm}
{\small
\begin{tabular}{|c|c|c|c|c||c|c|c|c|c|}
\hline
  Problem& $\gamma$ & Iter&Time& $\|F\|_{\infty}$&Problem& $\gamma$ & Iter&Time& $\|F\|_{\infty}$ \\ 
\hline     
1 &1 &6  & $5.5e$-$2$ &$2.7e$-$8$&15 &1 &5  &$3.6e$-$2$ &$1.5e$-$7$ \\ 
   &2 &4  & $1.5e$-$2$ &$1.2e$-$9$& &2 &7  &$1.1e$-$2$ &$4.6e$-$12$ \\ 
&3 &4  &$1.3e$-$2$ &$1.0e$-$7$ & &3 &9  &$1.2e$-$2$ &$9.1e$-$10$ \\ 
\hline    
   2 &1 &11 &$2.7e$-$2$ & $1.0e$-$7$&16&1&8  &$1.7e$-$2$ &$5.0e$-$7$ \\
      &2 &13  &$3.0e$-$2$ &$4.7e$-$7$&&2& 6 &$9.7e$-$3$ &$1.2e$-$12$ \\ 
      &3 &14  &$3.3e$-$2$ & $8.3e$-$7$ &&3& 9  &$1.1e$-$2$ &$4.6e$-$11$ \\  
\hline    
    4 &1 &4  &$6.6e$-$2$ &$1.2e$-$7$&17  &1 &3  &$1.1e$-$2$ &$1.9e$-$7$\\ 
      &2 &5  &$1.6e$-$2$ &$1.2e$-$9$&  & 2 &5  &$1.3e$-$2$ &$7.3e$-$10$ \\ 
      &3  &5  &$1.7e$-$2$ &$4.5e$-$13$&& 3 &9  &$7.4e$-$3$ &$4.8e$-$13$ \\  
 \hline    
      5 &1 &10 &$2.9e$-$2$ &$4.2e$-$8$& 18 & 1 &3  &$3.1$e-$2$ &$1.8e$-$8$ \\ 
      &2 &  *& &&&  2 &5  &$9.7e$-$3$ &$3.2$e-$14$ \\ 
      &2.5 & * & &  && 3 &8  &$1.0e$-$2$ &$2.1e$-$8$ \\     
\hline    
        6 &1  &5  & $1.4e$-$1$ &$6.3e$-$8$ &19&1 &3  &$3.9e$-$2$ &$7.9e$-$10$ \\ 
         &2.5 &5  & $2.2e$-$2$ &$1.3e$-$11$&& 2 &51  &$4.1e$-$2$ &$7.9e$-$10$\\ 
          &3 &5 &$1.0e$-$1$ &$1.0e$-$11$ &&3 &45  &$4.0e$-$2$ &$7.9e$-$10$ \\     
          \hline    
       8&1 &17 &$3.3e$-$2$ &$1.5e$-$8$&20& 1 &3  &$2.8e$-$2$ &$4.9e$-$12$  \\ 
      &2 &28 &$1.8e$-$1$ &$1.5e$-$8$&& 2 &6  &$9.4e$-$3$ &$5.1e$-$12$ \\ 
      &3 & 10  &$1.9e$-$2$ &$6.3e$-$12$ && 3  &11 &$1.4e$-$2$ &$5.1e$-$12$ \\        
      \hline    
            9 &1 &90  &$1.0e$-$1$ &$2.0e$-$7$ &21 &1 &20  &$4.6e$+$0$ &$6.5e$-$7$ \\ 
      &2 &  *& &&&2  &*&& \\
      &3 & 10  &$1.5e$-$2$ &$1.9e$-$12$ && 3 & *&&   \\     
   \hline    
      10 &1 &* & & &22 &1 &14 &$1.3e$+$1$ &$1.1e$-$7$ \\ 
      &2   &7  &$1.4$e-$2$ &$6.9e$-$8$&& 2.5 &16  &$1.5e$+$1$ &$1.9e$-$8$\\
      &3  &9  &$1.7e$-$2$ &$7.8e$-$7$&& 3 &20  &$1.9e$+$1$ &$1.3e$-$10$ \\           
           \hline    
    11 &1 &14  &$4.2e$-$2$ &$1.5e$-$10$ &23 &1 &5  &$4.5e$-$1$ &$4.0e$-$12$ \\ 
      &2 & 7  &$9.7e$-$3$ &$7.5e$-$12$&& 2 &6  &$5.2e$-$1$ &$3.8e$-$9$\\  
      &3 &9  &$1.5e$-$2$ &$3.4e$-$7$ && 3 &6  &$5.3e$-$1$ &$2.5e$-$10$ \\      
        \hline    
           12 &1 &24  &$4.5e$-$2$ &$5.9e$-$12$ &24 &1 &7  &$6.0e$-$1$ &$1.4e$-$7$ \\ 
      &2 &  6  &$9.3e$-$3$ &$6.0e$-$9$&  & 2 &9  &$7.4e$-$1$ &$6.7e$-$9$ \\
      &3 & 9  &$1.1e$-$2$ &$1.2e$-$7$ && 3 &7  &$6.0e$-$1$ &$4.8e$-$8$ \\   
   \hline    
    13 &1 &7  &$2.9e$-$2$ &$1.2e$-$8$ &25 &1 &18  &$1.5$e-$1$ &$4.3e$-$7$ \\ 
      &2 &6  &$9.5e$-$3$ &$1.5e$-$7$&&2 &20  &$1.0e$-$1$ &$7.9e$-$7$\\
      &3  &9  &$1.3e$-$2$ &$3.7e$-$8$&& 3 &21 &$4.5e$-$2$ &$6.8e$-$7$\\      
    \hline    
 14 &1 &5  &$4.3e$-$2$ &$6.9e$-$10$&&&&&  \\ 
   &2 & 8  &$1.7e$-$2$ &$6.0e$-$10$&&&&& \\ 
 &3 & 9  &$1.1e$-$2$ &$7.6e$-$9$&&&&&\\   
                
\hline  
\end{tabular}
}
\end{table}

From Table~2, we  see that 
our method successfully ended 63 times on a total of 75 runs which shows its robustness. For comparison purposes,
let us consider the performance of three methods analyzed in \cite{morini1},  namely,   Scaled Trust-Region Newton (STRN),  
 Active Set-Type Newton (ASTN) and  Inexact Gauss-Newton-Type (IGNT) methods introduced in \cite{morini1,Mangasarian1,sandra}, respectively.
 Analyzing the numerical results in our Table~2 and the ones in  Tables 2, 3, 4 and 5 in \cite[Section~4]{morini1} for the 24 common  problems,  the numbers of success on a total of 72 runs
  are 60, 58, 55  and 65, for Newton-CondG, STRN, ASTN and IGNT methods, respectively.
 These results  indicate that Newton-CondG method is as effective as the other  methods aforementioned
 for the  set of problems considered.
 Moreover, it is worth to point out that  the numbers of the function and Jacobian evaluations of Newton-CondG method are equal to number of iterations where as usual we do not take into account  the functions evaluations due to finite-difference approximations of the Jacobians. However, for STRN, ASTN and IGNT
 methods the number of the function evaluations are, in general, greater than the number of iterations.
This lower cost evaluations may reflect in computational savings. 
Therefore,  we may conclude the  applicability  and effectiveness of our method.


 
\section*{Conclusion}

In this paper, we considered the problem of solving a constrained  system of nonlinear equations. We  proposed and analyzed a method which consists  of a combination of Newton and conditional gradient methods. The convergence analysis was done via the majorant conditions technique, which allowed us to prove convergence results for different families of nonlinear functions. Under reasonable assumptions, we were able to provide a convergence radius of the method and  
establish  some convergence rate results. 
 In order to show the performance of our method, we carried out some numerical experiments  and comparisons with some other methods were presented. It would be interesting for future research to combine the conditional gradient method with other Newton-like methods.

 \end{document}